\documentclass[12pt]{amsart}

\usepackage{amssymb}
\swapnumbers        

\newtheorem{theorem}{Theorem}
\newtheorem{lem}[theorem]{Lemma}
\newtheorem{definition}[theorem]{Definition}
\newtheorem{icomment}[theorem]{Remark}

\newtheorem{sled}[theorem]{Corollary}

\begin{document}

\begin{center}
{\bf ON 
THE COMPLETENESS OF ORBITS
OF A POMMIEZ OPERATOR  \\
IN WEIGHTED (LF)-SPACES OF ENTIRE FUNCTIONS}
\end{center} 

\bigskip

\centerline {\bf Olga A. Ivanova, Sergej N. Melikhov}

\bigskip

\begin{center}{\small Vorovich Institute of Mathematics, Mechanics and Computer Sciences, \\
  Southern Federal University\\
  Mil'chakova st. 8a, 
  344090 Rostov on Don,  Russia\\
neo$_{-}$ivolga@mail.ru\\

Vorovich Institute of Mathematics, Mechanics and Computer Sciences,\\
  Southern Federal University\\
  Mil'chakova st. 8a,  344090 Rostov on Don,  Russia\\
  and Southern Mathematical Institute 
  of Vladikavkaz Science Center \\
of the Russ. Acad. Sci.\\
  Markus st. 22,  362027 Vladikavkaz,
  Russia\\
melih@math.rsu.ru
}
\end{center}



\bigskip

\noindent
{\bf Abstract} 

\noindent
{\small
We describe cyclic vectors for a Pommiez operator on a weighted
(LF)-space $E$ of entire functions.
The full description is obtained where $E$ is the
Laplace transform of the strong dual of the space of all germs 
of holomorphic
functions on a convex locally closed set
in the complex plane.
}


\section*{Introduction}

In \cite{p1}--\cite{p3} Pommiez studied consecutive remainders
for the Taylor series 
of analytic functions on the disk zentered at zero.
In \cite{p4} generalized Newton series expansions for
analytic functions were investigated. In \cite{p1}--\cite{p4} 
the following difference operator was used:
$$
D_z(f)(t):=\begin{cases}
\frac{f(t)-f(z)}{t-z}, & \text{$t\ne z$,}\\
f'(z),& \text{$t=z$.}
\end{cases}
$$
After the papers \cite{p1}--\cite{p4}, it has become to refer to
the operators $D_z$ as Pommiez operators.
The operator $D_0$ is called also {\it the backward shift operator}.
In \cite{Ufa} an interpolating functional was studied that
plays an important role in the theory of exponential series and 
convolution equations for analytic functions. 
This functional is an abstract version of Leont'ev's interpolating function. 
It is defined with the help of
a Pommiez operator $D_{0,g_0}$ acting continuously and linearly
in a weighted (LF)-space $E$ of entire 
functions. The operator $D_{0,g_0}$ is defined by
$$
D_{0,g_0}(f)(t):=\begin{cases}
\frac{f(t)-g_0(t)f(0)}{t}, & \text{$t\ne 0$,}\\
f'(0)-g_0'(0)f(0),& \text{$t=0.$}
\end{cases}
$$
Here $g_0$ is a fixed function in $E$ such that $g_0(0)=1$. 
If $g_0$ is identically equal to $1$, the operator $D_{0,g_0}$ 
coincides with the operator $D_0$.
If $E$ does not contain the function $g_0\equiv 1$, then
the standard operator $D_0$ may fail to act in $E$. 
Hence the passage to a function $g_0\in E$  with
$g_0(0)=1$ proves to be quite natural.
Tkachenko \cite{TKACH1}, \cite{TKACH2} 
investigated the operator of generalized integration, which is 
the dual map of $D_{0,g_0}$ in the case $g_0=e^P$ 
where $P$ is a polynomial.
This operator $D_{0,g_0}$ acts in a 
weighted (LB)-space
of entire functions whose growth is determined with
the help of a 
$\rho$-trigonometric convex function ($\rho>0$).

In the present article we study cyclic vectors for $D_{0,g_0}$ in $E$. 
A vector $x$ in a locally convex space $F$ is called cyclic
for a continuous linear operator $A$ in $F$
provided the linear span of the orbit $\{A^n(x)\,:\,n\ge 0\}$ is dense
in $F$. Earlier similar results were obtained 
in following situations.
Khaplanov \cite{Haplanov} found a sufficient cyclicity condition 
for a function for the operator $D_0$
in the space of all analytic functions
on the disk $|z|<R$. Kaz'min \cite{Kazmin} proved cyclicity 
criteria for a function for $D_0$ 
in the Fr\'echet space $H(G)$ of all 
analytic functions on a 
simply connected domain $G$ in $\mathbb C$ containing zero.
N.~Linchuk \cite{Linchuk_mama} investigated cyclic vectors 
for the operator $D_0$ 
in $H(G)$ for a finitely connected domain $G$ in $\mathbb C$. 
Yu.~Linchuk
\cite{Linchuk_sin_doch} studied cyclic vectors of some "one-dimensional
variation" \, of $D_0$ in $H(G)$. This operator 
is the operator
$D_{0,g_0}$ for some function $g_0\in H(G)$ (see \cite{IM_UFA15}).
In \cite{IM_UFA15} cyclic vectors were described for $D_{0,g_0}$ in $H(G)$ 
without additional assumptions of the paper 
\cite{Linchuk_sin_doch}.
Douglas, Shapiro and Shields \cite{DSS}
investigated cyclic vectors and invariant subspaces
for the  backward shift operator $D_0$ in the Hardy space $H^2$
on the unit disk. Cyclic elements for the generalized backward shift operator 
were described also by
Godefroy and Shapiro \cite{GS}.

In \S~1 we introduce the space $E$
and operators of our interest. In \S~2 we prove 
abstract sufficient and necessary conditions for the completeness  
of the system
$\{D_{0,g_0}^n(f):n\ge 0\}$ in $E$. In \S~3 we apply obtained 
results to
the space $E$, which is  the realization (with the help of the 
Laplace transform) 
of the strong dual
of the space $H(Q)$ of all germs of analytic functions on a convex 
locally closed set 
$Q\subset\mathbb C$. Main result here is Theorem 19, in which
the full description of cyclic vectors for $D_{0,g_0}$ in $E$ is obtained.
In the proof of Theorem 19 we use essentially Leont'ev's
interpolating function.
As a consequence we characterize 
proper closed $D_{0,g_0}$-invariant subspaces
of such space $E$ in the case that $g_0$ has no zeros
and proper closed ideals in the algebra $(H(Q), \ast)$
where $\ast$ is the Duhamel product in $H(Q)$.

\section{An auxiliary information}


In \cite{Ufa}, \cite{AA} a Pommiez operator associated with a 
function $g_0$
in a weighted
$(LF)$-space of entire functions was investigated. We mention a
necessary information from \cite{Ufa}, \cite{AA}. For a continuous
function $v:\mathbb C\to \mathbb R$, for a function $f:\mathbb C
\to \mathbb C$ we put
$$
p_v(f):=\sup\limits_{z\in \mathbb C}\frac{|f(z)|}{\exp v(z)}.
$$

Let $v_{n,k}:\mathbb C\to \mathbb R$ be continuous functions 
such that
$$
v_{n,k+1}\le v_{n,k}\le v_{n+1,k}, \ n,k\in\mathbb N.
$$
Define weighted Fr\'echet spaces
$$
E_n:=\{f\in H(\mathbb C)\,:\, p_{v_{n,k}}(f)<+\infty \ 
\forall k \in \mathbb N\},  \ n\in \mathbb N.
$$
Here $H(\mathbb C)$ is  the space of all entire functions on
$\mathbb C$ with the compact open topology.
Note that $E_n$ is embedded continuously in $E_{n+1}$ for each 
$n\in \mathbb N$.
Put $E:=\bigcup\limits_{n\in\mathbb N} E_n$ and we 
endow $E$ with the
topology of the inductive limit of the sequence of Fr\'echet spaces 
$E_n$, $n\in\mathbb N$, with respect to embeddings $E_n$ in
$E$ (see \cite[Ch.~III, \S~24]{MEIVOGT}): \, $E:=\mathop{\rm ind}\limits_{n
\rightarrow} E_n$.

In the sequel, we assume that the double sequence
$(v_{n,k})_{n,k\in\mathbb N}$ satisfies the following condition:

$$
\forall n  \ \exists m \ \forall k  \ \exists s \ \exists C\ge 0:
$$
\begin{equation}
\sup\limits_{|t-z|\le 1}v_{n,s}(t)+\ln(1+|z|)\le 
\inf\limits_{|t-z|\le 1} v_{m,k}(t)+C, \ z\in \mathbb C.
\end{equation}
Then the space $E$ is invariant under differentiation and translation,
and for each $n\in \mathbb N$ there exists $m\in
\mathbb N$ such that every bounded set in $E_n$ is relatively
compact in $E_m$ \cite[Remark 1]{Ufa}. Moreover, $E$ is
invariant under multiplication by the independent
variable.


We assume that $E$ contains a function that is not identically zero. 
Then there
exists a function $g_0\in E$ such that $g_0(0)=1$. For arbitrary
function $g_0\in E$ with $g_0(0)=1$ we define the Pommiez operator 
associated with $g_0$ by
$$
D_{0,g_0}(f)(t):=\left\{
\begin{array}{cc}
\frac{f(t)-g_0(t)f(0)}{t}, & t\ne 0,\\
f'(0)-g_0'(0)f(0), & t=0.
\end{array}
\right.
$$
The operator $D_{0,g_0}$ maps linearly and, by the closed graph
theorem, continuously $E$ into itself.

For a locally convex space $F$ we
denote by $\mathcal L(F)$ the space of all continuous 
linear operators from $F$ into itself.

Following \cite{Binderman}, \cite{Dimovskiy}, we introduce
{\it the
shift operator for the Pommiez operator} $D_{0,g_0}$ 
$$
T_z(f)(t):=\left\{
\begin{array}{cc}
\frac{tf(t)g_0(z)-zf(z)g_0(t)}{t-z}, & \, t\ne z,\\
zg_0(z)f'(z)-zf(z)g'_0(z)+f(z)g_0(z), & t=z,
\end{array}
\right.
$$
$f\in E$. By \cite{AA} $T_z\in \mathcal L(E)$ for all $z\in
\mathbb C$.

\medskip
If $g_0=e^P$, where $P$ is a polynomial, Tkachenko \cite{TKACH2} used 
the operator $T_z$ for the description of operators commuting
with an operator $J_P$ of generalized integration. 
The operator $J_P$
is the dual map of $D_{0,e^P}$ that acts in a weighted (LB)-space
of entire functions. The growth of functions of this
space is determined with the help a
$\rho$-trigonometric convex function ($\rho>0$).

We mention some auxiliary statements \cite[Lemma 7, Lemma 11, Lemma 12]{AA} 
(here Lemma 1, Lemma 3, Lemma 4, respectively). 
For $z\in\mathbb C$ the symbol $\delta_z$ denotes the delta-function
$\delta_z(f):=f(z)$.

\begin{lem} \label{l:179_2}
For each integer $n\ge 0$
$$
D_{0,g_0}^n(f)(z)=\varphi_n(T_z(f)),  \ \   f\in E, \ z\in \mathbb C,
$$
where $\varphi_0=\delta_0$ and for each $n\in \mathbb N$ there
exist $c_{k,n}\in \mathbb C$, $0\le k\le n-1$, such that
$$
\varphi_n(f)=\frac{1}{n!}f^{(n)}(0)+\sum\limits_{k=0}^{n-1}c_{k,n}f^{(k)}(0), \, 
f\in E.
$$
\end{lem}

\medskip
For a locally convex space $F$ by $F'$ we denote
the topological dual of $F$.
By $\tau:=\tau(E',E)$ we denote the Mackey topology in $E'$, 
that is, the
topology of uniform convergence on the family of all absolutely
convex $\sigma(E,E')$-compact subsets of $E$ 
\cite[Ch.~III, \S~3]{MEIVOGT}. 
Here $\sigma(E,E')$ is the weak topology in $E$
which is defined by the natural duality between $E$ and $E'$.

\begin{icomment}{\rm \label{i:28} (a) Since $E$ is embedded continuously 
in $H(\mathbb C)$
for each $k\ge 0$ and $t\in \mathbb C$ linear
functionals $f\mapsto f^{(k)}(t)$ are continuous on $E$. 
Hence all functionals $\varphi_n$, $n\ge 0$,  in Lemma 1 are
continuous on $E$.

\noindent
(b) If $f\in E$ and $\varphi_n(f)=0$ for all $n\ge 0$,
then $f^{(n)}(0)=0$ for all $n\ge 0$. Hence the function $f$ 
vanishes on $\mathbb C$. Therefore the system
$\{\varphi_n:\,n\ge 0\}$ is complete in $(E',\tau)$.
}
\end{icomment}

For $f\in E$ we put
$$
\widetilde f(t,z):=T_z(f)(t), \  t,z\in \mathbb C.
$$
The function $\tilde f$ is entire on $\mathbb C^2$.

\medskip
\begin{lem} If a net $\Psi_\mu\in E'$, $\mu\in\Delta$, 
converges to $\psi\in E'$ in $(E',\tau)$, then
for each function $f\in E$
$$
\lim\limits_{\mu\in\Delta}(\Psi_\mu)_z(\widetilde f(\cdot,z))= 
\psi_z(\widetilde f(\cdot,z)) \,\, \mbox{ and } \,\,
\lim\limits_{\mu\in\Delta}(\Psi_\mu)_t(\widetilde f(t,\cdot))= 
\psi_t(\widetilde f(t,\cdot)),
$$
where the limits are taken in $E$.
\end{lem}

\begin{lem} Let $\varphi_n$, $n\ge 0$, be the functionals defined by Lemma 1.
For each entire function $h$ on $\mathbb C^2$, for all integers
$j,k\ge 0$
$$
(\varphi_j)_z((\varphi_k)_t(h(t,z)))=(\varphi_k)_t((\varphi_j)_z(h(t,z))).
$$
\end{lem}


\medskip

Let $\mathcal K(D_{0,g_0})$ be the set of all operators $B\in \mathcal L(E)$ such that
$BD_{0,g_0}=D_{0,g_0}B$ on $E$. By \cite[Theorem 15]{AA} the following description of 
$\mathcal K(D_{0,g_0})$ holds.

\begin{theorem} (i) If $B\in\mathcal K(D_{0,g_0})$, there exists a unique $l\in E'$ such that
$B(f)(z)=l(T_z(f))$ for all $f\in E$ and $z\in\mathbb C$.

\noindent
(ii) For each $l\in E'$ the operator $B(f)(z)=l(T_z(f))$,
$f\in E$, $z\in\mathbb C$, belongs to $\mathcal K(D_{0,g_0})$.
\end{theorem}


\section{Abstract criteria of the cyclicity}

In this section we describe cyclic vectors for
$D_{0,g_0}$ in $E$. 


\begin{definition}
Let $F$ be a locally convex space. An element $x\in F$ is called
{\it a cyclic} 
vector for an operator
$A\in \mathcal L(F)$ if the system $\{A^n(x): \,n\ge 0\}$ is
complete in $F$, that is, the linear span of
$\{A^n(x): \,n\ge 0\}$ is dense in $F$.
\end{definition}

For a locally convex space $F$  and an operator $A\in \mathcal
L(F)$ the symbol ${\rm Cycl}_F(A)$ denotes the set of all cyclic
vectors of $A$ in $F$. Further $E$ is the space as in \S~1. We
will write ${\rm Cycl}(A)$ instead of ${\rm Cycl}_E(A)$.

\medskip
The following assertion is obvious.

\begin{lem}\label{l:1411_1430} Let $F$ be a locally convex space
and $A\in \mathcal L(F)$ be surjective.
Then the set ${\rm Cycl}_F(A)$ is $A$-invariant, that is,
$$
A({\rm Cycl}_F(A))\subseteq {\rm Cycl}_F(A).
$$
\end{lem}

\begin{sled}\label{s:1411_1425}
The set ${\rm Cycl}(D_{0,g_0})$ is $D_{0,g_0}$-invariant.
\end{sled}

\begin{proof}
The operator $D_{0,g_0}:E\to E$ is surjective. In
fact, for $f\in E$ the function $h(z)=zf(z)$, $z\in \mathbb C$,
belongs to $E$ and $D_{0,g_0}(h)=f$. By Lemma \ref{l:1411_1430}
$D_{0,g_0}({\rm Cycl}(D_{0,g_0}))\subseteq {\rm Cycl}
(D_{0,g_0})$. 
\end{proof}


\medskip
Put for $z\in\mathbb C$, $f\in E$
$$
\widetilde T_z(f)(t):=\left\{
\begin{array}{cc}
\frac{f(t)g_0(z)-f(z)g_0(t)}{t-z}, & t\ne z,\\
f'(z)g_0(z)-f(z)g_0'(z), & t\ne z,
\end{array}
\right.
$$
Clearly, $\tilde T_z\in{\mathcal L}(E)$ for all $z\in\mathbb C$.
Note that Krasichkov-Ternovskii \cite[\S~10]{KRTER3}  used 
similar constructions in some spaces of entire
functions of exponential type for the solution of a problem
on the extension of spectral synthesis.

\medskip

\begin{theorem}\label{th:811_1402}
The following assertions are equivalent:
\begin{enumerate}
\item [$(i)$] $f\in {\rm Cycl}(D_{0,g_0})$.
\item [$(ii)$] The
system $\{T_z(f):z\in \mathbb C\}$  is complete in $E$.
\item [$(iii)$] $f\notin {\rm Ker}\,B$ for each nonzero operator
$B\in \mathcal K(D_{0,g_0})$.
\item[(iv)] The
system $\{\widetilde T_z(f):z\in \mathbb C\}$  is complete in $E$.
\end{enumerate}
\end{theorem}
\begin{proof}
By Theorem 5 assertions $(ii)$ and $(iii)$ are equivalent.

$(i)\Rightarrow  (iii)$: Let $B\in {\mathcal K}(D_{0,g_0})$ be nonzero.
By Theorem 5 there exists $\varphi\in E'\backslash\{0\}$ such that
$B(h)(z):=l(T_z(h))$ for all $h\in E$ and $z\in\mathbb C$. 
We assume that $f\in{\rm Ker}\,B$.
Then for all $n\ge 0$
$$
0=D_0^n(B(f))=B(D_0^n(f)).
$$
Hence $D_0^n(f)\in {\rm Ker}\,B$ for each $n\ge 0$ and
${\rm{span}}\{D_0^n(f):n \ge 0\}\subseteq {\rm{Ker}}\,B$. Since $f$ 
is a cyclic vector for $D_{0,g_0}$, the closed subspace ${\rm Ker}\,B$
of $E$ coincides with $E$ and consequently
$B$ is zero operator. A contradiction.


$(ii)\Rightarrow (i)$: We take $l\in E'$ such that $l(D_{0,g_0}^n(f))=0$ for each
$n\ge 0$. Let $\varphi_n$, $n\ge 0$, be the functionals defined by Lemma 1.
Since the system $\{\varphi_n:n\ge 0\}$ is complete in
$(E',\tau)$, there exists a net
$\{\Phi_\alpha=\sum\limits_{j=0}^{m_\alpha}
a_{j\alpha}\varphi_j:\alpha\in \Lambda\}$ convergent to $l$ in
$(E',\tau)$.

Let as before $\widetilde f(t,z):=T_z(f)(t)$, $f\in E$, $t,z\in
\mathbb C$. By Lemma 1 
$$
D_{0,g_0}^n(f)(z)=(\varphi_n)_t(\widetilde f(t,z)), \ z\in \mathbb C, \, n\ge 0.
$$
Fix $n\ge 0$. Taking into account Lemma 3 and Lemma 4, we have:

$$
0=l(D_{0,g_0}^n(f))=l_z((\varphi_n)_t(\widetilde f(t,z)))=
\lim\limits_{\alpha\in \Lambda}(\Phi_\alpha)_z((\varphi_n)_t(\widetilde f(t,z)))=
$$
$$
\lim\limits_{\alpha\in \Lambda}\sum\limits_{j=0}^{m_\alpha}
a_{j\alpha}(\varphi_j)_z((\varphi_n)_t(\widetilde f(t,z)))
=\lim\limits_{\alpha\in \Lambda}\sum\limits_{j=0}^{m_\alpha}
a_{j\alpha}(\varphi_n)_t((\varphi_j)_z(\widetilde f(t,z)))=
$$
$$
\lim\limits_{\alpha\in \Lambda}(\varphi_n)_t\left(\sum\limits_{j=0}^{m_\alpha}
a_{j\alpha}(\varphi_j)_z(\widetilde f(t,z))\right)=
\lim\limits_{\alpha\in \Lambda}(\varphi_n)_t((\Phi_\alpha)_z(\widetilde f(t,z))=
$$
$$
(\varphi_n)_t(l_z(\widetilde f(t,z)))=(\varphi_n)_t(l_z(T_z(f)(t)))=
$$
$$
(\varphi_n)_t(l_z(T_t(f)(z)))=(\varphi_n)_t(l(T_t(f))).
$$
Hence for the entire function $h(t):=l(T_t(f))$ depending on $t$ we have 
$h^{(n)}(0)=0$ for all $n\ge 0$. Therefore $l(T_t(f))=0$
for all $t\in\mathbb C$, that, by hypothesis, implies that $l=0$.
Consequently, (i) holds.

$(ii)\Rightarrow (iv)$: Since for $t\ne
z$
$$
g_0(z)f(t)-\frac{tf(t)g_0(z)-zf(z)g_0(t)}{t-z}=-z\frac{f(t)g_0(z)-f(z)g_0(t)}{t-z},
$$
the equality
\begin{equation}\label{eq:811_1453}
g_0(z)f-T_z(f)=-z\widetilde T_z(f), \ z\in \mathbb C,
\end{equation}
holds. Fix $l\in E'$ such that $l(\widetilde T_z(f))=0$ for all $z\in
\mathbb C$. Acting by $l$ on the equality (2), we have for all $z\in\mathbb C$
$$
g_0(z)l(f)=l(T_z(f)).
$$
We put for $h\in E$ and $z\in\mathbb C$
$$
B(h)(z):=l(T_z(h)).
$$
By Theorem 5, $B\in \mathcal K(D_{0,g_0})$. Hence for
all $t\in \mathbb C$, since $D_{0,g_0}(g_0)=0$,
$$
0=D_{0,g_0}(l(f)g_0)(t)=D_{0,g_0}(B(f))(t)=
B(D_{0,g_0}(f))(t)=l(T_t(D_{0,g_0}(f))).
$$
By $(ii)\Rightarrow (i)$ the function $f$ is a cyclic vector 
for $D_{0,g_0}$ and, by
Corollary 8, $D_{0,g_0}(f)\in {\rm
Cycl}(D_{0,g_0})$. Consequently, by $(i)\Rightarrow (ii)$,
$l=0$. Therefore (iv) satisfies.

$(iv)\Rightarrow (ii)$: Fix $l\in E'$ such
that $l(T_z(f))=0$ for all $z\in \mathbb C$. If we act by $l$ on
the equality (2), we have for all $z\in\mathbb C$
$$
g_0(z)l(f)=-zl(\widetilde T_z(f)),
$$
that implies, since $g_0(0)=1$, that $l(f)=0$. Consequently, the entire
function $l(\widetilde T_z(f))$ depending on $z$ vanishes on $\mathbb C$.
By hypothesis, $l=0$. Hence (ii) is satisfied.
\end{proof}

\medskip

By $M$ we denote the operator of multiplication by the independent variable:
$$
M(f)(z):=zf(z), \, f\in E, \, z\in\mathbb C.
$$
It follows from (1) that $M\in{\mathcal L}(E)$. 
Denote by $I$ the identity mapping.

\begin{lem}\label{l:142_1539}
Assume that $g_0(\alpha)\ne 0$ for some $\alpha\in \mathbb C$.
Let $f\in E$. Consider assertions
\begin{enumerate}
\item [$(i)$]The system $\{T_z(f):z\in \mathbb C\}$ is complete in
$E$. 
\item [$(ii)$] The system $\{T_z((M-\alpha I)(f)):z\in
\mathbb C\}$ is complete in $E$. 
\end{enumerate}
Then $(i)\Rightarrow (ii)$.
\end{lem}

\begin{proof}
We use the equality
$$
tg_0(z)f(t)-\frac{t(t-\alpha)f(t)g_0(z)-z(z-\alpha)f(z)g_0(t)}{t-z}=
$$
$$
-(z-\alpha)\frac{tf(t)g_0(z)-zf(z)g_0(t)}{t-z}, \,\, t\ne z.
$$
This implies that for all $z\in\mathbb C$
\begin{equation}\label{eq:142_1601}
g_0(z)M(f) - T_z((M-\alpha I)(f))=-(z-\alpha)T_z(f).
\end{equation}

Assume that $(i)$ holds. We take
$l\in E'$ such that $l(T_z((M-\alpha I)(f)))=0$ for all $z\in
\mathbb C$. Acting by $l$ on the equality (3), we
have for all $z\in\mathbb C$
$$
g_0(z)l(M(f))=-(z-\alpha)l(T_z(f)).
$$
From $g_0(\alpha)\ne 0$ it follows that $l(M(f))=0$. Since  
$l(T_z(f))$ is
an entire function depending on $z$, then
$l(T_z(f))=0$ for all $z\in \mathbb
C$. By hypothesis, $l=0$. Hence (ii) satisfies.
\end{proof}

\begin{icomment}\label{i:53_2229}{\rm  If $g_0(\alpha)=0$ 
and $f(\alpha)=0$ for some $\alpha\in \mathbb C$ and some 
$f\in E$,
then $D_{0,g_0}^n(f)(\alpha)=0$ for
all $n\ge 0$.
}
\end{icomment}

\begin{sled}\label{s:142_1709} (i) Let $f\in {\rm Cycl}(D_{0,g_0})$
and $P$ be a polynomial. Then $Pf\in {\rm Cycl}(D_{0,g_0})$ if
and only if $P$ does not have common zeros with $g_0$.

\noindent
(ii) Assume that the function $g_0$ has no zeros 
in $\mathbb C$.
Then $Pf\in {\rm Cycl}(D_{0,g_0})$ for each 
$f\in {\rm Cycl}(D_{0,g_0})$ and for each nonzero 
polynomial
$P$.
\end{sled}
\begin{proof} (i): Let $Pf\in {\rm Cycl}(D_{0,g_0})$. 
We assume that $P$ and $g_0$ have a 
common zero $\alpha\in \mathbb C$. Then
${D_{0,g_0}^n(Pf)(\alpha)=0}$ for all $n\ge 0$. Therefore each
function in $F:={\rm span}\{D_{0,g_0}^n(Pf):\,n\ge 0\}$ and in the 
closure $F$ in $E$, which coincides with $E$,
vanishes at the point $\alpha$. A contradiction.

If $P$ and $g_0$ do not have common zeros the function
$Pf$ is a cyclic vector by Lemma 10.

(ii) follows from (i).
\end{proof}

\medskip

For $z\in\mathbb C$ the Pommiez operators $D_z$ 
are defined by
$$
D_z(f)(t):=\left\{
\begin{array}{cc}
\frac{f(t)-f(z)}{t-z}, & t\ne z,\\
f'(z),& t=z,
\end{array}
\right.
$$
for $f\in H(\mathbb C)$.
If $g_0\equiv 1\in E$, all operators $D_z$ map 
continuously and linearly $E$ into itself.

\begin{sled}\label{s:142_1721} Let $g_0\equiv 1$, $f\in E$.
The following statements are equivalent:
\begin{enumerate}
\item [$(i)$] $f\in {\rm Cycl}(D_{0})$. 
\item [$(ii)$] The system
$\{T_z(f):z\in \mathbb C\}$  is complete in в $E$. 
\item [$(iii)$]
The system $\{D_z(f):z\in \mathbb C\}$ is complete in $E$.
\end{enumerate}
\end{sled}

It follows from Theorem 9, since
$D_z(f)=\widetilde T_z(f)$ if $g_0\equiv 1$.


\medskip
Further we will use essentially the following relation between
operators $\tilde T_z$ and $D_z$.

\medskip
\begin{lem}
For all $f\in E$, $z\in\mathbb C$
$$
\tilde T_z(f)=g_0(z)D_z(f)-f(z)D_z(g_0).
$$
\end{lem}

(We note that $\tilde T_z(f)\in E$, but $g_0(z)D_z(f)$ and
$f(z)D_z(g_0)$ can need not belong to $E$.)

\medskip

This assertion can be verified directly.

\section{A concrete example}

Let $Q$ be a convex locally closed set in $\mathbb C$, 
that is,
a convex set having a countable fundamental sequence of
compact subsets (see \cite{MATHSCAND}, \cite{PUBL}). 
By \cite[Lemma 1.2]{MATHSCAND}  $Q$ is the union of
the relative interior of $Q$
and an open portion of the relative boundary of $Q$.
The family of such sets $Q$ contains all convex
domains in $\mathbb C$ and all convex compact sets in 
$\mathbb C$.
We assume that $0\in Q$. Let
$(Q_n)_{n\in\mathbb N}$ be an increasing fundamental 
sequence of compact sets in $Q$. 
Without loss
of generality all sets $Q_n$ are convex and $0\in Q_1$.
We consider the spaces $H(Q_n)$ of all
functions which are holomorphic on some neighborhood
of $Q_n$ with their natural inductive topology.
Let $H(Q)$ be the vector space of all holomorphic functions 
on $Q$, that is, 
holomorphic on some open neighborhood of $Q$. Since the 
algebraic equality 
$H(Q)=\bigcap\limits_{n\in\mathbb N} H(Q_n)$ holds
we endow $H(Q)$ with the projective 
topology of $H(Q):={\rm proj}_{\gets n}H(Q_n)$.

By $H_\Omega(z):=\sup\limits_{t\in \Omega}{\rm Re}(tz)$, 
$z\in \mathbb C$, 
we denote the support 
function of a set 
$\Omega\subset\mathbb C$. Put
$$
v_{n,k}(z):=H_{Q_n}(z)+|z|/k, \ z\in \mathbb C, \ n,\, k\in\mathbb N.
$$
All functions $v_{n,k}$ are subadditive and positively
homogeneous of degree $1$. The sequence $(v_{n,k})_{n,k\in\mathbb N}$
satisfies the condition (1).

We put $E_Q:=E$. Note that for each $n\in\mathbb N$
an entire function $h$ of exponential type 
belongs to $E_n$ if and only if
the conjugate diagram of $h$ is contained in $Q_n$.

By \cite[Lemma 1.10]{MATHSCAND} 
the Laplace transform
$$
\mathcal F(\varphi)(z):=\varphi(e^{\cdot z}), \,\,
\varphi\in H(Q)', \, z\in\mathbb C,
$$
is a topological isomorphism of the strong dual of $H(Q)$ 
onto the space $E_Q$. 
We will describe cyclic vectors for an operator $D_{0,g_0}$ in $E_Q$.

\subsection{An interpolating function}

Let $K$ be a convex compact set in $\mathbb C$ such that $0\in K$. 
For an entire function $f$ of exponential type 
with the conjugate diagram contained in $K$, 
for $x\in H(K)$ Leont'ev
(see \cite{Leontev}) introduced the interpolating function
$\omega_f:\mathbb C\times H(K) \to\mathbb C$ by
$$
\omega_f(z,x):=\frac{1}{2\pi i}
\int\limits_C\left(\int\limits_0^t x(t-\xi)e^{z\xi}d\xi\right)
\gamma_f(t)dt,
\,\,\, z\in\mathbb C,\,\, x\in H(K).
$$
Here $C$ is a closed convex curve which surrounds $K$, 
the function $x$
is holomorphic on the closure of the interior of the curve $C$,
$\gamma_f$ is the Borel transform of $f$.
The inner integral is taken along the 
segment $[0,t]$. Note that $\omega_f(z,x)$ is the entire 
function depending on $z$.

By \cite[Theorem 1]{Leontev} the following uniqueness theorem 
holds.

\medskip
\begin{theorem} Let $K$ be a convex compact set in $\mathbb C$ 
containing $0$, $f$ be an entire function of exponential type 
with the conjugate diagram contained in $K$.
If $f$ has infinitely many zeros and
for $x\in H(K)$ the function $\omega_f(z,x)/f(z)$ is entire 
one depending on $z$,
then $x=0$.
\end{theorem}

\medskip
Let $e_\nu(z):=e^{\nu z}$, \, $z, \nu\in\mathbb C$.

\begin{icomment} {\rm
(a)
Put
$$
\langle x,h\rangle:={\mathcal F}^{-1}(h)(x),\,\, x\in H(K),\,
h\in E_K.
$$
By \cite[Example 1]{Ufa}
$$
\langle x,D_z(f)\rangle=\omega_f(z,x), \,\, x\in H(K),\,
f\in E_K, \, z\in\mathbb C.
$$

\noindent
(b) For each $\nu\in K$, $j\ge 0$,
for the function $h(z):=z^j e^{\nu z}$
$$
\gamma_h(t)=j!(t-\nu)^{-j-1}.
$$

\noindent
(c) The function $e_\nu$ belongs to $E_Q$ if and only if
$\nu\in Q$. 

\noindent
(d) For each $\nu\in Q$, $x\in H(Q)$, $z\in\mathbb C$
$$
\omega_{e_\nu}(z,x)=\int\limits_0^\nu x(\nu-\xi)e^{z\xi}d\xi=
e^{\nu z}\int\limits_0^\nu x(\eta)e^{-z\eta}d\eta.
$$
}
\end{icomment}

\subsection{Cyclic elements of the Pommiez operator in $E_Q$}

\medskip
\begin{icomment} {\rm
(a) Since $0\in Q$, the function $h\equiv 1$ belongs to $E_Q$. Hence
$D_z\in{\mathcal L}(E_Q)$ for each $z\in\mathbb C$.

\noindent
(b) For each function $h\in E_Q$ the following alternative is valid:
either $h$ has infinitely many zeros or there exist $\lambda\in Q$ 
and a polynomial $P$ such that $h=Pe_\lambda$.
}
\end{icomment}

\medskip
\begin{lem}
Let $y$ be a continuous function on a segment 
$[\alpha,\beta]$ ($\alpha\neq\beta$)
and assume that the entire function 
$U(z)=e^{\alpha z}\int\limits_{\alpha}^{\beta} y(\eta)e^{-z\eta}d\eta$
is a polynomial. 
Then $y=0$ on $[\alpha,\beta]$.
\end{lem}

\begin{proof} We note that
$$
U(z)=e^{\alpha z}\int\limits_{\alpha}^{\beta} y(\eta)e^{-z\eta}d\eta=
\int\limits_{\alpha}^{\beta}y(\eta)e^{z(\alpha-\eta)}d\eta=
-\int\limits_0^{\alpha-\beta}y(\alpha-\xi)e^{z\xi}d\xi.
$$
Let $\alpha-\beta=\rho e^{i\varphi}$ with $\rho>0$
 and $\varphi\in\mathbb R$.
We make the change of the variable 
$\xi=re^{i\varphi}$,
where $r\in [0,\rho]$, and we take
$z$ in the form $z:=-te^{-i\varphi}$ with $t>0$. 
Then
$$
U(z)=-\int\limits_0^\rho
y(\alpha-re^{i\varphi})e^{-tr}e^{i\varphi}dr.
$$
Hence there exists $D>0$ such that
$$
|U(z)|\le D\int\limits_0^\rho e^{-tr}dr=-\frac{D}{t}(e^{-\rho t}-1).
$$
From here it follows that $|U(z)|\to 0$ as $t\to+\infty$.
Since $U$ is a polynomial, $U=0$ on $\mathbb C$.
From the completeness of the system $\{e_z\,:\,z\in\mathbb C\}$ 
in the space
of all continuous functions on the segment $[0,\rho]$ 
it follows that the function
$y$ vanishes on $[\alpha,\beta]$.
\end{proof}

\medskip
\begin{theorem}\label{th:11}
(I) Suppose that the function $g_0$ has infinitely many zeros.
The following assertions are equivalent:

\begin{itemize}
\item[(i)] $f\in {\rm Cycl}(D_{0,g_0})$.
\item[(ii)] Functions $f$ and $g_0$ do not have common zeros.
\end{itemize}

(II) Suppose that $g_0=Pe_\lambda$ for some $\lambda\in Q$ and 
for some polynomial 
$P$ with $P(0)=1$. The following assertions are equivalent:

\begin{itemize}
\item[(i)] $f\in {\rm Cycl}(D_{0,g_0})$.
\item[(iii)] Functions $f$ and $g_0$ do not have common zeros
and $f$ is not a function of the form $Re_\lambda$ 
where $R$ is a polynomial.
\end{itemize}

\end{theorem}

\begin{proof} (I) $(i)\Rightarrow(ii)$: We assume that $f$ and $g_0$ have a 
common zero $\alpha\in \mathbb C$. Proceeding as in the proof of (i) of
Corollary 12, we conclude that
each function in $E$ vanishes in $\alpha$. A contradiction.


\medskip
\noindent
$(ii)\Rightarrow(i)$: We take $n\in\mathbb N$ such that $g_0\in E_n$. 
The conjugate diagram of $g_0$ is contained in $Q_n$. 
Fix $x\in H(Q)\subset H(Q_n)$ such that 
$\langle x,\tilde T_z(f)\rangle=0$ 
for all $z\in\mathbb C$. By Lemma 14 for each $z\in\mathbb C$
$$
g_0(z)\langle x,D_z(f)\rangle=f(z)\langle x,D_z(g_0)\rangle.
$$
By Remark 16~(a) for all $z\in\mathbb C$
$$
g_0(z)\omega_f(z,x)=f(z)\omega_{g_0}(z,x).
$$
Since $g_0$ and $f$ do not have common zeros, $\omega_{g_0}(z,x)/g_0(z)$
is entire function depending on $z$. By Theorem 15, $x=0$. 
Consequently, the system
$\{\tilde T_z(f)\,:\,z\in\mathbb C\}$ is complete in $E_Q$.
By Theorem 9, $f\in {\rm Cycl}(D_{0,g_0})$.

\medskip
(II) $(iii)\Rightarrow(i)$: Consider all possible situations.

Suppose that $f$ has infinitely many zeros. 
Proceeding as in the proof of the implication 
$(ii)\Rightarrow(i)$, we obtain that 
$f\in {\rm Cycl}(D_{0,g_0})$.

Suppose now that $f$ has no zeros or $f$ has finite many zeros. Then
$f=Se_\mu$ for some $\mu\in Q$ such that $\mu\neq\lambda$ 
and for some
polynomial $S$. We fix $n\in\mathbb N$ such that
$\lambda, \mu\in Q_n$.

Let at first $S\equiv{\rm const}=C_0\neq 0$. Without loss of 
generality $C_0=1$.
Let
$P(z)=\sum\limits_{j=0}^m a_jz^j$ where $m\in\mathbb N$
and $a_0=1$.
We fix $x\in H(Q)$ with $\langle x,\tilde T_z(f)\rangle = 0$ 
for all $z\in\mathbb C$.

By Lemma 14 and Remark 16~(a) for all $z\in\mathbb C$
$$
g_0(z)\omega_f(z,x)=f(z)\omega_{g_0}(z,x).
$$

Let $m=1$.
By Remark 16~(d)
$$
\omega_f(z,x)=e^{\mu z}Y(\mu,z)
$$
where $Y(t,z):=\int\limits_0^t x(\eta)e^{-z\eta}d\eta$. Hence
\begin{equation}
g_0(z)Y(\mu, z)=\omega_{g_0}(z,x), \,\,z\in\mathbb C.
\end{equation}
By Remark 16~(b) for $K:=Q_n$, for some curve $C$ as in 3.1
$$
\omega_{g_0}(z,x)=e^{z\lambda}Y(\lambda,z) + a_1\frac{1}{2\pi i}\int\limits_C
\frac{1}{(t-\lambda)^2}\left(\int\limits_0^tx(t-\xi)e^{z\xi}d\xi\right)dt.
$$
Since 
$$
\int\limits_0^tx(t-\xi)e^{z\xi}d\xi=e^{zt}\int\limits_0^tx(\eta)e^{-z\eta}d\eta
$$
we have by the integral Cauchy formula
$$
\omega_{g_0}(z,x)=e^{z\lambda}Y(\lambda,z) + 
a_1\frac{d}{dt}\left(e^{zt}\int\limits_0^tx(\eta)e^{-z\eta}
d\eta\right)\Big|_{t=\lambda}=
$$
$$
e^{z\lambda}Y(\lambda,z)+a_1(ze^{\lambda z} Y(\lambda,z)+x(\lambda)).
$$
Thus by (4)
$$
e^{\lambda z}(1+a_1z)Y(\mu,z)=e^{z\lambda}Y(\lambda,z)+
a_1(ze^{\lambda z} Y(\lambda,z)+x(\lambda)).
$$
From here it follows that for all $z\in\mathbb C$
$$
e^{\lambda z}(1+a_1z)\int\limits_\lambda^\mu x(\eta)e^{-z\eta}
d\eta=a_1x(\lambda).
$$
By Lemma 18, $x=0$ on $[\lambda,\mu]$ and consequently $x=0$ 
as an element of $H(Q)$.

Let now $m\ge 2$. Then by Remark 16~(b) 
and by the integral Cauchy formula
$$
\omega_{g_0}(z,x)=e^{\lambda z}Y(\lambda,z)+
\sum\limits_{j=1}^ma_j\frac{j!}{2\pi i}
\int\limits_C
\frac{dt}{(t-\lambda)^{j+1}}\int\limits_0^tx(t-\xi)e^{z\xi}d\xi=
$$
$$
e^{\lambda z}Y(\lambda,z)+ \sum\limits_{j=1}^ma_j\frac{d^j}{dt^j}
\left(e^{zt}\int\limits_0^tx(\eta)e^{-z\eta}
d\eta\right)\Big|_{t=\lambda}=
$$
$$
e^{\lambda z}Y(\lambda,z)+ a_1(ze^{\lambda z}Y(\lambda,z)+x(\lambda))+
\sum\limits_{j=2}^m a_j\frac{d^j}{dt^j}
\left(e^{zt}\int\limits_0^tx(\eta)e^{-z\eta}
d\eta\right)\Big|_{t=\lambda}.
$$
For $j\ge 2$
$$
\frac{d^j}{dt^j}
\left(e^{zt}\int\limits_0^t x(\eta)e^{-z\eta}
d\eta\right)=\sum\limits_{s=0}^j C_j^sz^{j-s}e^{zt}\frac{d^s}{dt^s}\left(Y(t,z)\right)=
z^je^{zt}Y(t,z)+
$$
$$
C_j^1z^{j-1}e^{zt}x(t)e^{-zt}+
\sum\limits_{s=2}^jC_j^sz^{j-s}e^{zt}\sum\limits_{r=0}^{s-1}C_{s-1}^r (-1)^rz^r
e^{-zt}x^{(s-1-r)}(t).
$$
Consequently,
$$
\omega_{g_0}(z,x)=e^{\lambda z}Y(\lambda,z)+a_1(ze^{\lambda z} Y(\lambda,z)+x(\lambda))+
$$
$$
\sum\limits_{j=2}^m a_j\left(z^je^{\lambda z}Y(\lambda,z)
+ 
\right.
$$
$$
\left. 
C_j^1z^{j-1}x(\lambda)
+\sum\limits_{s=2}^jC_j^sz^{j-s}e^{\lambda z}
\sum\limits_{r=0}^{s-1}C_{s-1}^r(-1)^rz^r
e^{-\lambda z}x^{(s-1-r)}(\lambda)\right)=
$$
$$
e^{\lambda z}Y(\lambda,z)+a_1ze^{\lambda z} Y(\lambda,z)+ 
\sum\limits_{j=2}^m a_jz^je^{\lambda z}Y(\lambda,z)+W(z)=
$$
$$
e^{\lambda z}Y(\lambda,z)g_0(z) + W(z)
$$
where $W$ is a polynomial depending on $z$. It has the form
$$
W(z)=\sum\limits_{p=0}^{m-1} w_p(z)x^{(p)}(\lambda)
$$ 
where
polynomials $w_p$ do not depend on $x\in H(Q)$.
Hence by (4) for all $z\in\mathbb C$
$$
e^{\lambda z}\left(1+
\sum\limits_{j=1}^ma_jz^j\right)\int\limits_\lambda^\mu x(\eta)e^{-z\eta}d\eta=
W(z).
$$ 
From here it follows that the entire function
$e^{\lambda z}\int\limits_\lambda^\mu x(\eta)e^{-z\eta}d\eta$ 
depending on $z$
is a polynomial.
By Lemma 18 $x=0$.
Thus the system $\{\tilde T_z\,:\,z\in\mathbb C\}$ is complete in $E_Q$ and
by Theorem 9 $f\in {\rm Cycl}(D_{0,g_0})$.

Finally, for arbitrary polynomial $S$ which does not have common zeros with $P$ 
the assertion follows from Corollary 12~(i).

\medskip
$(i)\Rightarrow(iii)$: From Remark 11 it follows that
$f$ and $g_0$ do not have common zeros.
We assume that there exists
a polynomial $R$ of degree at most $s\in\mathbb N$ 
such that $f=Re_\lambda$. Let
$P(z):=\sum\limits_{j=0}^ma_j z^j$, $m\in\mathbb N$.
As in the proof of the implication
$(iii)\Rightarrow(i)$ we conclude 
that there are polynomials $w_p$, $v_q$ such that
for all $x\in H(Q)$
$$
\omega_{g_0}(z,x)=P(z)e^{\lambda z}Y(\lambda,z) + W(z)
$$
and
$$
\omega_f(z,x)=R(z)e^{\lambda z} Y(\lambda,z) + V(z)
$$
where
$W(z)=\sum\limits_{p=0}^{m-1}x^{(p)}(\lambda) w_p(z)$
and
$V(z)=\sum\limits_{q=0}^{s-1}x^{(q)}(\lambda)v_q(z)$.
There exists a nonzero function $x\in H(Q)$ such that
$x^{(r)}(\lambda)=0$ if $0\le r\le{\rm max}(m,s)$.
Then
$W=V=0$ and for all $z\in\mathbb C$
$$
P(z)e^{\lambda z}(R(z)e^{\lambda z} Y(\lambda,z) + V(z))=
R(z)e^{\lambda z}(P(z)e^{\lambda z}Y(\lambda,z)+W(z)),
$$
that is, $g_0(z)\omega_f(z,x)=f(z)\omega_{g_0}(z,x)$
for all $z\in\mathbb C$. Consequently, by Remark 16~(a)
$$
g_0(z)\langle x,D_z(f)\rangle=f(z)\langle x,D_z(g_0)\rangle
$$
and, by Lemma 14, $\langle x,\tilde T_z(f)\rangle=0$ 
for all $z\in\mathbb C$. Hence the system
$\{\tilde T_z(f)\,:\,z\in\mathbb C\}$ is 
incomplete in $E$ and by Theorem 9
$f\notin{\rm Cycl}(D_{0,g_0})$. A contradiction.
\end{proof}


\subsection{Invariant subspaces of the Pommiez operator 
in $E_Q$}

\medskip
We apply Theorem 19 to a description of proper
closed $D_{0,g_0}$-invariant
subspaces of $E_Q$ where $g_0$ has no zeros.

Denote by $\mathbb C[z]$ (resp. $\mathbb C[z]_n$ for integer $n\ge 0$)
the space of all
polynomials (resp. of degree at most $n$). 
Put for integer $n\ge 0$ and $\lambda\in\mathbb C$ 
$$
{\mathcal P}_n(e_\lambda):= e_\lambda\cdot\mathbb C[z]_n:=
\{e_\lambda P\,:\, P\in C[z]_n\}.
$$

\medskip
\begin{sled} Let $g_0=e_\lambda$ for some 
$\lambda\in Q$.

\begin{itemize}
\item[(i)] For each integer 
$n\ge 0$ the space ${\mathcal P}_n(e_\lambda)$ is a proper closed
$D_{0,g_0}$-invariant subspace of $E_Q$.
\item[(ii)] For each proper closed
$D_{0,g_0}$-invariant subspace $\mathcal P$ of $E_Q$ there is $n\ge 0$ such that
$\mathcal P={\mathcal P}_n(e_\lambda)$.
\end{itemize}
\end{sled}

\begin{proof} (i): Each space $\mathcal P_n(e_\lambda)$ is finite-dimensional and hence
closed. Clearly $\mathcal P_n(e_\lambda)\ne E_Q$ and
$\mathcal P_n(e_\lambda)$ is $D_{0,g_0}$-invariant.

(ii):  Let $\mathcal P$ be a proper closed $D_{0,g_0}$-invariant
subspace of $E_Q$. By Theorem 19 $\mathcal P\subset e_\lambda\cdot\mathbb C[z]$.
Show that
$$
n(\mathcal P):={\rm sup}\{{\rm deg}(P)\,:\,Pe_\lambda\in\mathcal P\}<+\infty.
$$ 
We assume that $n(\mathcal P)=+\infty$.
Note that for each polynom $S$ of degree $m\in\mathbb N$, 
for the function $h=Se_\lambda$ 
the function $D_{0,g_0}(h)$	has the form $S_1e_\lambda$ 
where $S_1$ is a polynomial 
of degree $m-1$. Hence for each $m\ge 0$ there exists a polynomial $R_m$ such that
${\rm deg}(R_m)=m$ and $R_m e_\lambda\in\mathcal P$. 
From here it follows that
$e_\lambda\cdot\mathbb C[z]\subset\mathcal P$.

Let $\lambda\in Q_n$. We take $p\in\mathbb N$ and $h\in E_p$.
By \cite[Theorem 4.4]{KRASTERN} there are polynomials
$R_j$, $j\in\mathbb N$, such that $R_je_\lambda\to h$ in $E_q$
where $q={\rm max}(n, p)$.
Hence the subspace $e_\lambda\cdot\mathbb C[z]$ is dense
in $E_Q$. Consequently, $\mathcal P=E_Q$. A contradiction.												
Thus $n(\mathcal P)<+\infty$ and 
$\mathcal P=\mathcal P_{n(\mathcal P)}(e_\lambda)$.
\end{proof}

\subsection{Ideals in the algebra $(H(Q), \ast)$}

Let $E$  be the space as in \S~1. 
It is possible to connect 
with the operator $D_{0,g_0}$ a multiplication
in $E'$. By \cite{AA} for all $\varphi, \psi\in E'$ the multiplication 
(convolution) $\varphi\otimes\psi$ is defined by
$$
\varphi\otimes\psi(f):=\varphi_z(\psi(T_z(f))), \,\, f\in E.
$$
The operation $\otimes$ is associative and commutative.
The vector space $E'$ with this multiplication is an algebra. 
For a set $T\subset E$ (resp. $T\subset E'$)
we denote by $T^0$ the polar of $T$ in
$E'$ (resp. in $E$). We endow $E'$ with the Mackey
topology $\tau(E',E)$. The following 
{\it duality principle} holds.

\begin{lem} The following assertions are equivalent:

\begin{itemize}
\item[(i)] A set $L\subset E'$ is a proper closed ideal
in $(E',\otimes)$.
\item[(ii)]There exists a proper closed $D_{0,g_0}$-invariant
subspace $\mathcal P$ 
of $E$ such that $L={\mathcal P}^0$.
\end{itemize}
\end{lem}

\begin{proof} $(i)\Rightarrow(ii)$:
Let $L$ be a proper closed ideal in $(E',\otimes)$.
Fix $f\in L^0$. 
Then $\psi\otimes\varphi(f)=0$ for all
$\psi\in L$ and $\varphi\in E'$.
By Lemma 1 $D_{0,g_0}(f)(z)=\varphi_1(T_z(f))$ for all 
$z\in\mathbb C$.
Hence for all $\psi\in L$
$$
\psi(D_{0,g_0}(f))=\psi_z(\varphi_1(T_z(f)))=
\psi\otimes\varphi_1(f)=0.
$$
Consequently, $D_{0,g_0}(f)\in L^0$,
and $\mathcal P:=L^0$ is a proper closed
$D_{0,g_0}$-invariant subspace of $E$.
Since $L$ is a closed subspace of $E'$, by bipolar theorem 
\cite[Ch.~III, 22.13]{MEIVOGT}
$L=L^{00}={\mathcal P}^0$.

$(ii)\Rightarrow(i)$: Let $L={\mathcal P}^0$ for some proper
closed $D_{0,g_0}$-invariant subspace $\mathcal P$ of $E$.
Fix $\varphi\in E'$, $\psi\in L$. By 
\cite[the proof of Theorem 14]{AA} for all $z\in\mathbb C$
there exists a net
$B_{\alpha,z}=\sum\limits_{j=0}^{n_{\alpha,z}}b_{j,\alpha, z}D_{0,g_0}^j$,  
$\alpha\in\Lambda_z$, such that
$\lim\limits_{\alpha\in\Lambda_z}B_{\alpha,z}(f)=T_z(f)$ in $E$.
Then for all $f\in\mathcal P$
$$
\varphi\otimes\psi(f)=
\varphi_z(\psi(T_z(f)))=\varphi_z\left(\psi\left(
\lim\limits_{\alpha\in\Lambda_z}
\sum\limits_{j=0}^{n_{\alpha,z}}b_{j,\alpha, z}D_{0,g_0}^j(f)\right)\right)=
$$
$$
\varphi_z\left(
\lim\limits_{\alpha\in\Lambda_z}
\sum\limits_{j=0}^{n_{\alpha,z}}b_{j,\alpha, z}
\psi\left(D_{0,g_0}^j(f)\right)\right)=0
$$
since $\psi(D_{0,g_0}^j(f))=0$ for all $j\ge 0$.
Consequently, $\varphi\otimes\psi\in {\mathcal P}^0$,
and $L$ is a proper closed ideal in $(E',\otimes)$.

\end{proof}


\medskip

We return to the space $E_Q$. By \cite[Proposition 1.9]{MART}
the projective topology in $H(Q)$ coincides with the inductive topology
of ${\rm ind}_\Omega H(\Omega)$ where $\Omega$ runs over all open 
neighborhoods $\Omega$
of $Q$ and where $H(\Omega)$ is the Fr\'echet space of all holomorphic 
functions on $\Omega$. By \cite[\S~3, p.~65]{MART} $H(Q)$ is a Montel space.
Hence \cite[Remark 24.24~(a)]{MEIVOGT} the space $H(Q)$ is reflexive,
and $E'_Q$ can be identified 
with the space $H(Q)$. In addition, the transform 
$$
\mathcal J: E'_Q\to H(Q), \,\, \varphi \mapsto \varphi(e_z), 
$$
is a topological isomorphism of the strong dual of $E_Q$
onto $H(Q)$. Put $\widehat{\varphi}:={\mathcal J}(\varphi)$, \, 
$\varphi\in E'_Q$.

From \cite[Example 3.2]{AA} it follows that for $g_0\equiv 1$
the multiplication $\otimes$ in $E'_Q$ is realized in $H(Q)$ 
in the following way. 
For all  $\varphi,\psi\in E'_Q$
$$
\mathcal J(\varphi\otimes\psi)(z)=
\widehat \psi(0)\widehat \varphi(z)+\int\limits_0^z
\widehat \varphi(\xi)(\widehat \psi) '(z-\xi)d\xi
$$
where the integral is taken along the 
segment $[0,z]$, and $z$ belongs to the union
of some convex domains $G_n$, $n\in\mathbb N$,
such that $Q_n\subset G_n$ and functions $\widehat{\varphi}$,
$\widehat{\psi}$ are holomorphic on each $G_n$.
Hence for $g_0\equiv 1$ the operation $\otimes$ 
is realized as 
the Duhamel product $\ast$ in $H(Q)$:
$$
v\ast w(z)=w(0)v(z)+\int\limits_0^z
v(\xi)w'(z-\xi)d\xi, \,\, v, w\in H(Q).
$$

The space $H(Q)$ is an algebra with the multiplication $\ast$.
We will describe proper closed ideals in $(H(Q),\ast)$. 
We note at first that for all $\lambda\in Q$, 
$h\in H(Q)$ and integer $n\ge 0$
the following equality holds:
$$
{\mathcal J}^{-1}(h)_z(z^ne^{\lambda z})=h^{(n)}(\lambda).
$$

From Corollary 20 and Lemma  21 it follows

\begin{sled} The following assertions are equivalent:
\begin{itemize}
\item[(i)] $T$ is a proper closed ideal in $(H(Q),\ast)$.
\item[(ii)] There exists an integer $n\ge 0$ such that
$$
T=\{h\in H(Q) \, | \,
h^{(j)}(0)=0, \,\, 0\le j\le n\}.
$$
\end{itemize}

\end{sled}

\medskip
{\bf Acknowledgement}.
The authors would like to express gratitude to Professor
J.~Bonet for fruitful discussions and valuable comments.

\end{document}